\documentclass[11pt,leqno]{amsart}
\usepackage{amsmath}
\usepackage{amsfonts}
\usepackage{amssymb}
\usepackage{graphicx}
\usepackage{hyperref}
\providecommand{\U}[1]{\protect\rule{.1in}{.1in}}

\newtheorem{theorem}{Theorem}

\newtheorem{corollary}[theorem]{Corollary}

\newtheorem{lemma}[theorem]{Lemma}

\newtheorem{remark}[theorem]{Remark}

\newcommand{\tr}{\operatorname{tr}}	
\newcommand{\Wex}{W^{1,p}_{\mathrm{ex}}(M,N)}
\newcommand{\Win}{W^{1,p}_{\mathrm{in}}(M,N)}
\newcommand{\Wexf}{W^{1,p}_{\mathrm{ex}}(M,N,f)}
\newcommand{\Winf}{W^{1,p}_{\mathrm{in}}(M,N,f)}

	\newcommand{\interior}{\mathrm{int}}
	\newcommand{\Lip}{\operatorname{Lip}}
	\newcommand{\Jac}{\operatorname{Jac}}

\newcommand{\ball}{B^{\mathbb{R}^n}_R}
\newcommand{\ballr}[1]{B^{\mathbb{R}^n}_{#1}}

\newcommand{\rr}{\mathbb{R}}
\newcommand{\hh}{\mathbb{H}}
\newcommand{\hj}{{\hat{j}}}
\newcommand{\Hess}{\operatorname{Hess}}

\newcommand{\sect}{\operatorname{Sect}}

\newcommand{\vol}{\operatorname{Vol}}
\newcommand{\area}{\mathrm{Area}}

\newcommand{\R}{\mathbb{R}}

\newcommand{\dist}{\operatorname{dist}}

\begin{document}

\title[Dirichlet problem for harmonic maps] {Sobolev spaces of maps and the Dirichlet problem for harmonic maps}
\date{\today}
\author{Stefano Pigola}
\address{Dipartimento di Scienza e Alta Tecnologia,
Universit\`a degli Studi dell'Insubria, \\
via Valleggio 11, 
I-22100 Como, ITALY}
\email{stefano.pigola@uninsubria.it}

\author{Giona Veronelli}
\address{
Universit\'e Paris 13, Sorbonne Paris Cit\'e, LAGA, CNRS ( UMR 7539)\\
99, avenue Jean-Baptiste Cl\'ement F-93430 Villetaneuse - FRANCE }
\email{veronelli@math.univ-paris13.fr}



\begin{abstract}In this paper we prove the existence of a solution to the Dirichlet problem for harmonic maps into a geodesic ball on which the squared distance function from the origin is strictly convex. This improves a celebrated theorem obtained by S. Hildebrandt, H. Kaul and K. Widman in 1977. In particular no curvature assumptions on the target are required. Our proof relies on a careful analysis of the Sobolev spaces of maps involved in the variational process, and on a deformation result which permits to glue a suitable Euclidean end to the geodesic ball.\end{abstract}

\maketitle

\tableofcontents

\section{Introduction and main result}

Let $(M,g)$ be a compact, connected, $m$-dimensional Riemannian manifold with
smooth boundary $\partial M\neq\emptyset$ and let $(N,h)$ be a complete,
connected, $n$-dimensional Riemannian manifold without boundary. Having fixed
a (sufficiently regular) boundary datum $f:\partial M\rightarrow N$ the
corresponding Dirichlet problem consists in finding a (sufficiently regular)
map $u:M\rightarrow N$ which extends $f$ to a harmonic map on $\mathrm{int}M$.
Recall that a map $u:\interior M\rightarrow N$ is called harmonic
if it is a stationary point of the
energy functional $E_{2}(v)=\frac{1}{2}\int_M \|dv\|^{2}_{\mathrm{HS}}d\vol$ with respect to intrinsic variations keeping the boundary
values $v|_{\partial M}=f$ fixed. In this expression, the differential $dv$ is considered as a section of the bundle $T^{\ast}M \otimes v^{-1}TN$ so that $\|dv\|_{\mathrm{HS}}= \sqrt{\tr_M h(dv,dv)}$ represents its Hilbert-Schmidt norm.
The  Euler-Lagrange equations satisfied by the stationary point $u$ of the energy functional read $\Delta u =0$ on $\mathrm{int}M$, where the manifold-valued Laplacian, in local coordinates, has an expression like $(\vec{\Delta u})^A = \Delta_M \vec{u}^A + \vec{\Gamma}^A(D\vec{u},D\vec{u})$. Here, $\Delta_M$ denotes the Laplace-Beltrami operator, $D$ stands for the differentiation of vector-valued functions and  $\vec{\Gamma}^A$ is a family of  smooth quadratic forms depending on the metric $g$ and on the Christoffel symbols of $N$. We shall refer to $\Delta u =0$ as the harmonic mapping system.\\

The Dirichlet problem for harmonic maps between Riemannian manifolds or, more generally, for metric spaces with different geometric structures, has attracted the attention of both analysts and geometers since the appearance of the foundational paper \cite{ES-AmerJour} by J. Eells and J. Sampson. Indeed, if on the one hand the problem in itself has an analytic flavour, it is immediately seen that the energy functional of a map and the corresponding Euler-Lagrange equations are directly related to the geometry of the domain and the target spaces. It follows that the solvability of the problem, the regularity of a given solution, as well as the way one is naturally led to follow to build such a solution, are very much influenced by the geometric structure of the spaces.

In order to study the solvability of the Dirichlet problem for harmonic maps we have two possible approaches to our disposal: (a) the parabolic approach, that relies on the analysis of the heat flow associated to the manifold-valued Laplacian and (b) the elliptic approach, that relies on the direct calculus of variations for the Dirichlet energy functional.

The parabolic approach, although very powerful, typically requires non-positivity of the curvature of the target space in order to obtain large-time existence of the heat flow and the corresponding convergence to a harmonic map. This restriction can be relaxed via the elliptic approach.  In this setting, as remarked also by F. H\'{e}lein and J. Wood, \cite{HW-handbook}, the most general results known in the literature (in case of trivial topology) require that the range of the boundary datum is contained in a \textit{regular ball}. This means that the ball lies within the cut-locus of its center and the radius $R$ is related to an upper bound $\Lambda >0$ of the sectional curvatures by the inequality $2R\sqrt{\Lambda}<\pi$. This curvature restriction is vital for the known proofs to work; see Appendix \ref{appendix-known} for a detailed analysis.\\

In the present paper, we shall use the purely elliptic viewpoint to obtain a solution of the Dirichlet problem without imposing any curvature condition on the target space.
\begin{theorem}\label{main_th}
Let $\left(  M,g\right)  $ be a compact, $m$-dimensional Riemannian manifold
with smooth boundary $\partial M\neq\emptyset$ and let $B_{R}^{N}\left(
q_{0}\right) $, $0<R\leq +\infty$, be a geodesic ball into the complete $n$-dimensional
Riemannian manifold $\left(  N,h\right)  $. Assume that the function $\varrho(q)=\dist^2_N(q,q_0)$ is smooth and strictly convex on $B^N_R(q_0)$. Then, for any
given $f\in C^{0}\left(  M,B_{R}^{N}\left(  q_{0}\right)  \right)  \cap
\Lip\left(  \partial M,B_{R}^{N}\left(  q_{0}\right)  \right)  $, the
Dirichlet problem 
\begin{equation}
\left\{
\begin{array}
[c]{ll}
\Delta u=0 & \text{on }M\\
u=f & \text{on }\partial M,
\end{array}
\right.  \label{p-pb}
\end{equation}
has a solution $u\in C^{\infty
}\left(  \operatorname{int}(M),B_{R}^{N}\left(  q_{0}\right)  \right)  \cap
C^{0}\left(  M,B_{R}^{N}\left(  q_{0}\right)  \right)  $.
\end{theorem}

Obviously, the smoothness assumption on the distance function implies that $B^N_R(q_0) \cap \mathrm{cut}(q_0) = \emptyset$. However, it is readily seen in concrete examples that the convexity assumption does not imply any curvature restriction at the points of the same ball. The next simple example should clarify the situation.

Having fixed a smooth
function $\sigma:[0,+\infty)\rightarrow\lbrack0,+\infty)$ satisfying%
\begin{equation}
\sigma^{(2k)}\left(  0\right)  =0\text{, }\forall k \in \mathbb{N},\quad\sigma^{\prime}\left(  0\right)
=1,\quad\sigma\left(  r\right)  >0\text{, }\forall r>0,\label{model1}
\end{equation}
we shall denote by $N_{\sigma}^{n}$ the smooth $n$-dimensional Riemannian
manifold given by
\begin{equation}
\left(  \lbrack0,+\infty)\times\mathbb{S}^{n-1},dr\otimes dr+\sigma^{2}\left(
r\right)  g_{\mathbb S^{n-1}}\right) .\label{model2}
\end{equation}
Clearly, $N_{\sigma}^{n}$ is diffeomorphic to $\mathbb{R}^{n}$ and geodesically complete for any choice of $\sigma$. Usually, $N_{\sigma}^{n}$ is
called a model manifold with warping function $\sigma$ and pole $0$. The $r$-coordinate in the expression (\ref{model2}) of the metric represents the
distance from the pole. Thus, at a given point of $N_{\sigma}^{n}$ we distinguish the radial sectional curvatures and the tangential sectional
curvatures of the model, according to whether the $2$-plane at hand contains the radial vector field $\partial/\partial r$ or not. Standard formulas for
warped product metrics show that
\begin{equation}\label{sec_rot}
\sect_{rad}=-\frac{\sigma^{\prime\prime}}{\sigma},\quad \sect_{tg}=\frac{1-\left(
\sigma^{\prime}\right)  ^{2}}{\sigma^{2}}
\end{equation}
and
\begin{equation}
\mathrm{Hess}(r^2) = 2 dr\otimes dr  + 2r \sigma' (r)\sigma(r) g_{\mathbb S^{n-1}}.
\end{equation}
In particular, the smooth function $\varrho(x)=r^2(x)$ is strictly convex provided $\sigma'(r)>0$. Now, using the gluing process described in \cite{PV-pDirichlet-symmetry}, we can choose $\sigma$ to be a smooth, increasing, concave function such that $\sigma(r)=\sin(r)$ for $0 \leq r \leq \pi/4$ and $\sigma(r)=ar+b$ for every $r \gg 1$ and for appropriate $a,b>0$. Thus $\varrho$ is smooth and convex on $N^n_{\sigma}$ but, according to \eqref{sec_rot}, $\sup_{N^n_{\sigma}} \sect \geq 1$. In particular, the regularity of the ball $B^N_R(0)$ is violated for any $R \gg 1$.

Our strategy to prove Theorem \ref{main_th} consists in applying the direct calculus of variations once the target space is flattened out outside a large ball without violating the smooth convexity property of the distance function. This enables us to use Euclidean methods, such as energy-decreasing projections along the radial direction, and to get an unconstrained minimizer with prescribed boundary data. The standard maximum principle then shows that, actually, the range of the minimizer is inside the original ball and satisfies the harmonic mapping system in the original metric. In particular, it gives a solution of the Dirichlet problem. The minimization procedure, the radial projection argument, and the regularity theory applied to the minimizer all require an appropriate choice of the Sobolev space of maps. Since the target space is covered by a global system of normal coordinates one has a natural notion of ``intrinsic" Sobolev class of maps as opposed to the ``extrinsic" one which is obtained via an isometric embedding of the target into an Euclidean space. It turns out that these notions are both useful in the solution of the problem and very often  they are used interchangeably without too much care. Therefore,  a careful analysis on their reciprocal relations is needed. We take the occasion to shad a light on this delicate point at least for bounded maps. An account of the intrinsic metric viewpoint on the subject is presented in Appendix \ref{appendix-metricspaces}.\\

From the viewpoint of the methodology, this paper can be considered the third of a series of papers, \cite{PV-pDirichlet-compact, PV-pDirichlet-symmetry}, devoted to the study of solutions of the Dirichlet problem for $p$-harmonic mappings into different targets. In the first of these papers, we considered the relative homotopy Dirichlet problem for compact targets of non-positive curvature, working out the program initiated by B. White in \cite{Wh-Acta}. In the second paper we considered Cartan-Hadamard targets with rotational symmetry and, in this setting, we introduced some new geometric methods to reduce the problem to the compact case. Some of these methods are implemented and extended in the present (third) paper where, as alluded to above, we shall focus on the case where the range of the boundary datum is contained in a ball where the squared distance function from its center is smooth and strictly convex.

\section{Equivalent definitions of (bounded) Sobolev maps}\label{Sobolev-spaces}

The use of the direct calculus of variations to get minimizers of the
energy functional for manifold valued maps requires the introduction of a
suitable notion of Sobolev maps. Now, when we consider an $n$-dimensional
 target space $\left(  N,h\right)  $ which is covered by a global, normal coordinate chart centered at some point $q_0$,
 the most natural Sobolev
class of maps to work with is the one defined intrinsically, i.e., by using
the fact that $\left(  N,h\right)  $ can be identified with the Euclidean
space $\mathbb{R}^{n}$ endowed with a metric that can be expressed globally in
normal cartesian coordinates. Assuming that this identification is made once and for all then, for any  complete Riemannian domain $(M,g)$ (with possibly non-empty boundary), a map $v:M \to N$ is identified with its vector-valued representation $\vec v =(v^1,...,v^n):M \to \mathbb{R}^n$. Consequently, the intrinsic Sobolev space is defined by
\[
W^{1,p}_{\mathrm{in}}(M,N)=W^{1,p}(M,\mathbb{R}^n)
\]
the right hand side being understood as the closure of $C^{\infty}_c(M,\mathbb{R}^n)$ in the usual $W^{1,p}$ norm:
\begin{align*}
\|v\|_{W^{1,p}_\mathrm{int}}^p = \|  v \|_{L^p(M,\rr^n)}^p + \| D v \|_{L^p(M,\rr^n)}^p
\end{align*}
where, we recall, $D$ stands for the differentiation of a vector-valued map.\\

Regrettably, the regularity
theory developed by Schoen-Uhlenbeck, \cite{ScUh-JDG, ScUh-JDG2}, later extended to constrained maps and to the $p$-harmonic realm by Fuchs, \cite{Fu-Carolina, Fu1-Annali, Fu2-Annali}, Hardt-Lin, \cite{HL-CPAM} and Luckhaus, \cite{Lu-Crelle}, makes use of a notion of a Sobolev map which has a very extrinsic nature, namely, it requires the manifold at hand $(N,h)$ to be
embedded into some Euclidean space via an isometric embedding $i:N \hookrightarrow \mathbb{R}^k$. According to J. Nash famous theorem this is
always possible and, moreover, if $(N,h)$ is complete, we can also guarantee that the embedding is proper so that  $i(N)$ is a closed subset of $\mathbb{R}^k$, \cite{Mu-JMAA}. For complete targets we shall always tacitly assume that $i$ is a proper embedding. Thus, given a map $v:M \to N$, if we put
\[
\check{v} = i \circ v :M \to \mathbb{R}^k,
\]
in this extrinsic setting one defines
\begin{align}
W^{1,p}_{\mathrm{ex}}(M,N) &= \{ v:M \to N : \check{v} \in W^{1,p}(M,\mathbb{R}^k)\nonumber \\
	&= \{w \in W^{1,p}(M,\mathbb{R}^k): w(x) \in i(N) \text{ a.e.} \} \nonumber.
\end{align}
Even if the space $W^{1,p}_{\mathrm{ex}}(M,N)$ depends strongly on the embedding whenever $N$ is non-compact, in view of our purposes we shall omit this dependence in the notation and assume that such an embedding is fixed once and for all. Note that no explicit control on the extrinsic geometry of the submanifold $i(N)$ is available regardless of the fact that, from the
intrinsic viewpoint, the manifold  has a very simple structure. Thus, if
on the one hand, the extrinsic viewpoint is completely general and overcome
any possible difficulty derived from the topological complexity of the target
space, on the other hand it introduces unnatural difficulties of extrinsic
nature even in the simplest case described above (think of a Cartan-Hadamard target).
\smallskip

The main result of this section states that there is no difference between the intrinsic and extrinsic Sobolev classes of a bounded maps. In the next section we shall show that a similar conclusion holds in case the Sobolev maps have prescribed boundary values (in the trace sense).

\begin{theorem}\label{th_equiv-sob_1}
Let $(M,g)$ and $(N,h)$ be complete Riemannian manifolds with, possibily, $\partial M \not= \emptyset$ and $N$ endowed with a global normal coordinate chart centered at $q_0$. Assume that a proper, isometric embedding $i:N \hookrightarrow \mathbb{R}^k$ is chosen. Let $v:M \to B^N_{R_0}(q_0)\subset N$, for some $R_0>0$.
Then
\[
v \in W^{1,p}_{\mathrm{ex}}(M,N) \text{ if and only if } v\in W^{1,p}_{\mathrm{in}}(M,N).
\]
\end{theorem}
To this purpose, we preliminarily observe that the implication
\[
v \in W^{1,p}_{\mathrm{in}}(M,N) \Rightarrow v\in W^{1,p}_{\mathrm{ex}}(M,N)
\]
is quite natural. Indeed, for any compactly supported smooth map $w:M\to B^N_{R_0}(q_0) \Subset \rr^n$, since the embedding $i$ is isometric and the Riemannian metric $h$ is equivalent to the Euclidean metric on $\bar{B}^N_{2R_0}(q_0)$, we have
\begin{equation}\label{gradient-norms}
\|D w\|_{L^p(M,\mathbb{R}^n)}^p \approx \int_M \| dw \|_{\mathrm{HS}(M,N)}^p d\mathrm{vol}=  \| D\check{w}\|^p_{L^p{(M,\mathbb{R}^k)}},
\end{equation}
where the symbol $\approx$ means that the ratio of the two quantities is in between two positive constants, uniformly with respect to $w$. Moreover, if $v:M \to B^N_{R_0}(q_0) \Subset \rr^n$ is a second map,
\[
\dist_N(w(p),v(p)) \geq \dist_{\rr^k}(\check{w}(p), \check{v}(p))
\]
and, therefore,
\begin{equation}\label{dist-norms}
\| w - v\|_{L^p(M,\mathbb{R}^n)} \approx \| \dist_N(w,v)\|_{L^p(M,\rr)} \geq \| \check{w} - \check{v}\|_{L^p{(M,\mathbb{R}^k)}}.
\end{equation}
Accordingly, one is led to conclude the validity of the stated implication using the definition of the Sobolev space $W^{1,p}_{\mathrm{int}}(M,N)$ via density and the weak compactness of bounded subsets of $W^{1,p}(M,\rr^k)$.

As for the opposite implication, in view of  \eqref{gradient-norms}, the main point concerns with the relation between the extrinsic and the intrinsic $L^p$-norms of compactly supported approximating maps. Since the range of all maps is confined into a fixed compact set $\Omega$ in the embedded
target, it is reasonable that extrinsic and intrinsic distances in
$\Omega\times\Omega$ are equivalent, that is, the function $\alpha(p_1,p_2)=\dist_{N}\left(  p_1,p_2\right)/\dist_{\rr^k}(i(p_1),i(p_2))$ is bounded from above (and strictly positive) on $\Omega \times \Omega$. This obviously follows once we show that $\alpha$ is continuous. The only possible problems could occur along the diagonal, but near these points (recall that the submanifold is properly embedded) the two distances look asymptotically the same.\\

We are going to formalize a proof of both the implications stated in Theorem \ref{th_equiv-sob_1} using a somewhat more synthetic argument. This relies on a related result which is interesting in its own and will be crucial when boundary data in the trace sense are imposed; see Section \ref{section-f-sobolev}. Roughly speaking, we shall prove that we can flatten out the ball $i(B^N_{R_0}(q_0)) \subset \rr^k$ via a bi-Lipschitz ambient diffeomorphism that does not change the Euclidean space outside a large ball.

\begin{theorem} \label{th_flatten-ball} Let $(N,h)$ be a complete Riemannian manifold properly embedded into some Euclidean space via the isometric embedding
$i:N \hookrightarrow \mathbb{R}^k$.
Let $R_{1}>R_{0}>0$ be fixed radii with $R_0<\mathrm{inj}(q_0)$. Then, there exists
a smooth  diffeomorphism $F:\mathbb{R}^{k}\rightarrow
\mathbb{R}^{k}$ satisfying the following properties:

\begin{enumerate}
\item[(a)] $F\circ i=\mathrm{id}_{\mathbb{R}^{n}}\times0_{k-n}$ on $B_{R_{0}}
^{N}(q_0)  .$

\item[(b)] $F=\mathrm{id}_{\mathbb{R}^{k}}$ on $\mathbb{R}^{k}\backslash
B_{R_{1}}^{\mathbb{R}^{k}}\left(  0_{k}\right)  $.
\end{enumerate}
\end{theorem}

\begin{proof}
The result is essentially due to R. Palais, \cite[Theorem C]{P-extendingPAMS},
so we will only sketch the arguments adapted to our situation.

In the normal coordinate system of $N$ centered at $q_0$, we can identify $B^N_{R_0}(q_0)$ with the Euclidean ball $B^{\mathbb{R}^n}_{R_0}(0_n) \times 0_{k-n} \subset \mathbb{R}^k$ endowed with the metric $h$.
Without loss of generality, up to composing $i$ with an isometry of $\mathbb{R}^{k}$ (possibly
orientation reversing), we can also assume that $i\left(  0_{n}\right)
=0_{k}\in\mathbb{R}^{k}$, $T_{0_k}i\left(  N\right)  =\mathbb{R}^{n}\times0_{k-n}$ and
\[
i_{\ast0_{n}}=\mathrm{id}_{\rr^n}\times 0_{k-n} :T_{0_{n}}N=\mathbb{R}^{n}\rightarrow
T_{0}i\left(  N\right)  =\mathbb{R}^{n}\times0_{k-n}.
\]
Let $0<\delta \ll 1$ and consider the diffeomorphism
\[
\begin{array}
[c]{cccc}
f=\left.  i\right\vert _{B_{R_{0}+\delta}^{N}\left(  0_{n}\right)  }
^{-1}\times0_{k-n}: & \mathcal{B}_{R_{0}+\delta} & \rightarrow &
B_{R_{0}+\delta}^{N}\left(  0_{n}\right)  \times0_{k-n}\\
& \cap &  & \cap\\
& B_{R_{1}}^{\mathbb{R}^{k}}\left(  0_{k}\right)  &  & B_{R_{1}}
^{\mathbb{R}^{k}}\left(  0_{k}\right)
\end{array}
\]
where we have set $\mathcal{B}_R=i(B^N_R(q_0))$.
It is clear that any diffeomorphism $F:\mathbb{R}^{k}\rightarrow\mathbb{R}%
^{k}$ that extends $f$ and is the identity map outside $B_{R_{1}}%
^{\mathbb{R}^{k}}\left(  0_{k}\right)  $ has the desired properties. As in
\cite[Theorem A]{P-extendingPAMS}, it is easier to deal with $f^{-1}%
:B_{R_{0}+\delta}^{N}\left(  0_{n}\right)  \times0_{k-n}\rightarrow
\mathcal{B}_{R_{0}+\delta}$.

First of all, since the normal bundle of the contractible space $\mathcal{B}$ is trivial, the corresponding normal exponential map $\exp^{\perp}$ can be used to define a diffeomorphism
\[
\begin{array}
[c]{cccc}%
\varphi : & B_{R_{0}+\delta}^{N}\left(  0_{n}\right)  \times\left(
-\delta,\delta\right)  ^{k-n} & \rightarrow & \mathcal{N}_{\delta}\left(
\mathcal{B}_{R_{0}+\delta}\right) \\
& \cap &  & \cap\\
& B_{R_{1}}^{\mathbb{R}^{k}}\left(  0_{k}\right)  &  & B_{R_{1}}%
^{\mathbb{R}^{k}}\left(  0_{k}\right)
\end{array}
\]
by setting
\[
\varphi (x,y)  =\exp^{\bot}( f^{-1}(  x),y) .
\]
Here $\mathcal{N}_{\delta}\left(
\mathcal{B}_{R_{0}+\delta}\right)$ is the restriction of a tubular neighborhood of $i(N)$ to $\mathcal{B}_{R_{0}+\delta}$. Up to chose $\delta$ small enough, the restriction of the tubular neighborhood can be chosen uniform of radius $\delta$, and the diffeomorphism $\varphi$ is onto $\mathcal{N}_{\delta}\left(
\mathcal{B}_{R_{0}+\delta}\right)$.

Thus, $\varphi$ extends $f^{-1}$
outside $B_{R_{0}+\delta}^{N}\left(  0_{n}\right)  \times0_{k-n}$, and
satisfies $\varphi\left(  0_{k}\right)  =0_{k}$ and $\mathrm{Jac}\left(
\varphi\right)  \left(  0\right)  =\mathrm{I}_{k}$.

Next, according to \cite[Lemma 5.1]{P-natural-TAMS}, we pick a small
real number $0< \delta^{\prime} \ll 1$ in such a way that there exists a diffeomorphism
\[
H:\mathbb{R}^{k}\rightarrow\mathbb{R}^{k}=\mathbb{R}^n\times\mathbb{R}^{k-n}
\]
satisfying the following properties:

\begin{itemize}
\item $H=\varphi$ on $B_{\delta^{\prime}}^{N}\left(  0_{n}\right)
\times\left(  -\delta^{\prime},\delta^{\prime}\right)  ^{k-n};$

\item $H\left(  y\right)  =\mathrm{id}_{\mathbb{R}^{k}}\left(  y\right)  $ on
$\mathbb{R}^{k}\setminus B_{R_{1}}^{\mathbb{R}^{k}}\left(  0_{k}\right)  $.
\end{itemize}

The desired diffeomorphism $F:\mathbb{R}^{q}\rightarrow\mathbb{R}^{q}$ is finally
obtained by applying to $H$ the construction in \cite[Theorem B]{P-extendingPAMS}. This amounts to take the composition  $F = F^2 \circ H \circ F^1$ with ambient diffeomorphisms $F^1, F^2: \rr^k \to \rr^k$ that, respectively, shrink into and expand outside  $B_{\delta^{\prime}}^{N}\left(  0_{n}\right) \times (-\delta',\delta')$  a normal tubular neighborhood of $\mathcal B_{R_0 + \delta/4}$ and, moreover, that leave $\mathbb{R}^{k}\setminus B_{R_{1}}^{\mathbb{R}^{k}}\left(  0_{k}\right)  $ fixed.

\end{proof}

We observe that the map $F$ constructed in the previous theorem well behaves with respect to $\Wex$-convergence. This fact will be used systematically, therefore we state it as a formal lemma, which can be proved by standard computation.

\begin{lemma}\label{lemma_Fconv}
Let $(M,g)$ be a complete, $m$-dimensional Riemannian manifold possibly with $\partial M \not =\emptyset$, and let $F:\rr^k \to \rr^k$ be a $C^1$-map such that $\sup_{\rr^k} \| \Jac F\|_{\rr^{k^2}}  <+\infty.$ Let $\{ v_j\} \subset W^{1,p}(M,\rr^k)$ be a sequence satisfying $v_j \to v \in W^{1,p}(M,\rr^k)$ in $W^{1,p}(M,\rr^k)$. Then, there exists a subsequence $\{v_{j'}\}$ such that $F(v_{j'}) \to F(v) \text{ in } W^{1,p}(M,\rr^k).$
\end{lemma}

We are now in the position to give a very direct proof of Theorem  \ref{th_equiv-sob_1}.

\begin{proof}[Proof (of Theorem \ref{th_equiv-sob_1})]
Let $v\in \Wex$ be
such that $v(x)\in B_{R_0}^N(q_0)$ for a.e. $x\in M$. Let $\{w_{j}\}\subset C_c^{\infty}(M,\mathbb{R}^{k})$ be a sequence of smooth maps
satisfying $w_{j}\rightarrow\check{v}$ in $W^{1,p}(M,\mathbb{R}^{k})$
as $j\rightarrow\infty$. Let $F:\rr^k\to \rr^k$ be the diffeomorphism defined in
Theorem \ref{th_flatten-ball}. By Lemma \ref{lemma_Fconv}, up to extracting subsequences, we can assume that $F(w_j) \to F(\check{v})$ in $W^{1,p}(M,\rr^k)$ and a.e. in $M$. In view of the properties of $F$,
\[
F\circ\check{v}(x) = v(x) \times 0_{k-n}
\text{ a.e.} \ x\in M.
\]
Hence, if we denote by $\pi_n:\rr^k=\rr^n \times \rr^{k-n} \to \rr^n$ the standard orthogonal projection onto the first $n$-coordinates, we have
\[
v=\pi_{n}\circ F \circ\check{v}\in \Win = W^{1,p}(M,\rr^n).
\]

Conversely, let $v \in \Win$ be such that $v(x)\in B_{R}^N(q_0)$ for a.e. $x\in M$. We consider a sequence $\{v_j\} \subset C^{\infty}_c(M,\rr^n)$ satisfying $v_{j}\rightarrow\ v $ in $W^{1,p}(M,\mathbb{R}^{n})$ and we note that $v_j \times 0_{k-n} \to v\times 0_{k-n} $ in $W^{1,p}(M,\rr^k)$. Therefore, by Lemma \ref{lemma_Fconv} applied to $F^{-1}$, we conclude that, up to extracting a subsequence, $F^{-1}(v_j \times 0_{k-n}) \to F^{-1}(v\times 0_{k-n})$  in $W^{1,p}(M,\rr^k)$. In particular, by the properties of $F$, $$\check{v}(x) = F^{-1}(v(x)\times 0_{k-n}) \in W^{1,p}(M,\rr^k)=\Wex.$$ This completes the proof.
\end{proof}

We conclude this section by observing the following approximation result.
	\begin{corollary}
	Let $(M,g)$ and $(N,h)$ be complete Riemannian manifolds with, possibily, $\partial M \not= \emptyset$ and $N$ endowed with a global normal coordinate chart centered at $q_0$. Assume that a proper, isometric embedding $i:N \hookrightarrow \mathbb{R}^k$ is chosen. Let $v\in\Wex$ such that $v(x)\in B^N_R(q_0)\subset N$ for a.e. $x\in M$, for some $R>0$. Then there exists a sequence $\{v_j\}\subset C^\infty(M,N)$ such that $v_j\to v$ in $\Wex$ as $j\to\infty$.
	\end{corollary}

\section {Sobolev maps with boundary trace conditions}\label{section-f-sobolev}
Throughout this section $(M,g)$ will always denote an $m$-dimensional compact Riemannian manifold with boundary $\partial M \not= \emptyset$ whereas $(N,h)$ will be a complete $n$-dimensional Riemannian manifold without boundary and covered by a fixed global, normal coordinate chart centered at $q_0$. As above, using these coordinates, one identifies $N$, once and for all, with the Euclidean space $\rr^n$. In particular, according to the definitions in the previous section, the intrinsic Sobolev space $\Win$ is defined as $W^{1,p}(M,\rr^n)$. 

Now, the classical trace theory ensure the existence of a bounded linear operator
\[
\tr:W^{1,p}(M,\rr^n) \to L^p(\partial M, \rr^n).
\]
Therefore, given any map $f \in \Win$, its intrinsic boundary values $\tr f$ are well defined. Moreover, if $f$ is continuous, $\tr f = f|_{\partial M}$ is exactly the restriction of $f$ to the boundary set $\partial M$. It is also worth to recall that, given a second map $v \in \Win$, the condition $\tr v= \tr f$ can be stated equivalently by saying that
\[
v-f \in W^{1,p}_0(M,\rr^n),
\]
where $W^{1,p}_0(M,\rr^n)$ denotes the closure in the $W^{1,p}$ norm of the space $C^{\infty}_c(\mathrm{int}M,\rr^n)$.
Thus, we can define the intrinsic Sobolev space of $W^{1,p}$ maps with boundary values $f$ as
\[
\Winf = \{v\in \Win: \tr v = \tr f\}.
\]
Similarly, the corresponding extrinsic Sobolev space can be introduced:
\[
\Wexf = \{v\in \Wex: \tr \check{v} = \tr \check{f}\}.
\]
Observe that, since $f \in \Win \subset \Wex$ this definition makes sense.

With this preparation, the main result of the section states that for bounded maps, at least in case $f \in \Lip(M,N)$, these two classes coincide; see Corollary \ref{th_equiv-sob_3}. According to Theorem \ref{th_equiv-sob_1} this fact relies on the following
\begin{theorem}\label{th_equiv-sob_2}
Let $v: M \to N$ be a bounded map in $\Win$ or, equivalently, in  $\Wex$, and let $f \in \mathrm{Lip}(M,N)$. Then
\[
\tr v = \tr f \text{ in } W^{1,p}(M,\rr^n) \text{ if and only if } \tr \check{v}=\tr \check{f} \text{ in } W^{1,p}(M,\rr^k).
\]
\end{theorem}
\begin{proof}
Let $R>0$ be so large that $v(M), f(M) \subset B^N_R(q_0)\subset \rr^n$.

First, we assume that
\[
\tr v = \tr f \text{ in } W^{1,p}(M,\rr^n) .
\]
Then, there exists a sequence $\{w_j\} \subset C^{\infty}_c(\mathrm{int}M,\rr^n)$ such that $w_j \to v-f$ in the space $W^{1,p}(M,\rr^n)$. In particular, the sequence of continuous functions $v_j := w_j + f$ satisfies
\[
v_j \to v \text{ in } \Win=W^{1,p}(M,\rr^n)
\]
and, up to extracting a subsequence, we can assume that the convergence is pointwise a.e. In particular, we can assume that $v_j(M) \subset \bar{B}^N_{2R}(q_0)$.\\
Let $F : \rr^k \to \rr^k$ be the corresponding bi-Lipschitz diffeomorphism defined in Theorem \ref{th_flatten-ball} with $R_0=2R$.
We note that $v_j \times 0_{k-n} \to v\times 0_{k-n} $ in $W^{1,p}(M,\rr^k)$. Therefore, by Lemma \ref{lemma_Fconv} applied to $F^{-1}$, we conclude that, up to extracting a subsequence, 
$$F^{-1}(v_j \times 0_{k-n}) \to F^{-1}(v\times 0_{k-n})\text{ in } W^{1,p}(M,\rr^k).$$ 
By the properties of $F$, this latter is equivalent to \[
\check{v}_j \to \check{v} \text{ in } W^{1,p}(M,\rr^k).
\]
Moreover, recalling that $f$ and $v_j$ are both continuous,
\[
\tr v_j = v_j|_{\partial M} =  f|_{\partial M}.
\]
Since for continuous functions

\[
\tr \circ\; i = i \circ \tr,
\]
using the continuity of the trace operator we  conclude that
\[
\tr \check{v} = \lim_{L^p} \tr \check{v}_j = \lim_{L^p} (i \circ \tr v_j) = \check{f}|_{\partial M},
\]
as required. \smallskip

The proof of the opposite implication is similar. Indeed, assume that
\[
\tr \check{v} = \tr \check{f} \text{ in } W^{1,p}(M,\rr^k) .
\]
and, as above, take a sequence $w_j \in W^{1,p}(M,\rr^k) \cap C^0(M,\rr^k)$ satisfying
\[
w_j  \to \check{v} \text{ in } W^{1,p}(M,\rr^k)
\]
and
\[
\tr w_j = w_j|_{\partial M} = \check{f}|_{\partial M}.
\]
By Lemma \ref{lemma_Fconv}, up to extracting a subsequence,
\[
F(w_j) \to F(\check{v}) = v \times 0_{k-n} \text{ in } W^{1,p}(M,\rr^k)
\]
and, moreover,
\[
F(w_j )|_{\partial M} = F (\check{f}|_{\partial M}) = f|_{\partial M} \times 0_{k-n}.
\]
Whence, using once again the continuity of trace operator we conclude that 
\[
\tr (v \times 0_{k-n}) = (\tr v ) \times 0_{k-n} = f|_{\partial M} \times 0_{k-n}.
\]
This completes the proof of the Theorem.
\end{proof}

As an immediate consequence we obtain the announced equivalence between the extrinsic and the intrinsic notion of bounded Sobolev maps with trace boundary values. 
\begin{corollary}\label{th_equiv-sob_3}
Let $v: M \to N$ be a bounded map and let $f \in \mathrm{Lip}(M,N)$. Then
\[
v \in \Wexf \text{ if and only if } v\in \Winf.
\]
\end{corollary}

\section{Convex exhaustion functions on glued manifolds}
Suppose we are given a (bounded) geodesic ball, where the squared distance function from the origin is smooth and convex. The aim of this section is to extend the Riemannian metric outside the ball in such a way that the resulting manifold is complete, isometric to the Euclidean space outside a larger ball and, finally, with the property that the squared distance function from the same origin remains smooth and strictly convex.

\medskip

Let $B^{\mathbb{R}^n}_R$ be the Euclidean ball of radius $0<R \leq +\infty$ and centered at the origin.  Assume that on $\ball$ it is given a Riemannian metric that, with respect to fixed local polar coordinates $(t,\theta^2,...,\theta^n)$ on $\rr^n$, has the expression
\begin{equation}\label{metric-symm}
h=dt\otimes dt+\sum_{i,l=2}^{n}\sigma_{il}(t,\theta^2,...,\theta^n) d\theta^i\otimes d\theta^l.
\end{equation}
Since we are assuming that the above expression holds for $0<t<R$ by the very definition of $h$ and the Gauss Lemma we are implicitly asserting that $t$ is the distance function from the origin and $\mathrm{inj}(0)=R$. In
particular, for each $\theta\in\mathbb{S}^{n-1}$, $\gamma\left(  r\right)
=\theta r$ are minimizing geodesics.

As a special case, if $R=+\infty$, $N^n = (\mathbb{R}^{n},h)$ is a complete manifold with a pole
$0\in\mathbb{R}^{n}$. The compact geodesic balls $\bar{B}_{R}^{N^n}\left(
0\right) $ of $N$, i.e., the sublevel sets $\left\{  t(x)\leq R\right\}  $, are
diffeomorphic to the corresponding compact Euclidean balls $\bar{B}%
_{R}^{\mathbb{R}^{n}}$.

\medskip

Suppose that the squared distance function from the origin is smooth and strictly convex on $(\ball,h)$. In order to give it a Euclidean end, we implement a construction in two steps.
First, in Theorem \ref{th_gluing_3} we will show how to deform the metric of $h$ outside a smaller ball of radius $R_1<R$ in such a way that the new
metric is complete and eventually rotationally symmetric (in fact hyperbolic), and that the
squared radial function remains a globally defined strictly convex function.  
This deformation allows us to win the lack of rotational symmetry of the metric in the ball at hand. The problem of  obtaining a new manifold with an Euclidean end and with a strictly convex squared distance function will thus reduce to a one-dimensional problem, solved in Theorem \ref{th_gluing_4}.

\begin{theorem}
\label{th_gluing_3}Let $\ball$ be endowed with the Riemannian metric $h$ that, in polar coordinates, takes the form (\ref{metric-symm}). Assume that the smooth function $t^2:\ball \to [0,R^2)$ is convex. Fix $0<R_{1} < R_{2}< R$. Then, there exist $k>0$ depending on $h$, $R_{1}$ and $R_{2}$,
and a manifold $(\hat N^{n},\hat h)
=(\mathbb{R}^{n},dt^{2}+\hat{\sigma}_{il}(t,\theta^2,...,\theta^n)d\theta^{i}\otimes d\theta^{l})$ such that:

\begin{enumerate}
\item[(i)] $\hat{\sigma}_{il}(t,\theta^2,...,\theta^n)d\theta^{i}\otimes d\theta^{l}={\sigma}_{il}(t,\theta^2,...,\theta^n)d\theta^{i}\otimes d\theta^{l}$ when $t<R_1$.

\item[(ii)] $\hat{\sigma}_{il}(t,\theta^2,...\theta^n)d\theta^{i}\otimes d\theta^{l}=k^{-1}\sinh
^{2}(\sqrt{k}t)g_{\mathbb S^{n-1}}(\theta^2,...,\theta^n)$ when $t \geq R_2$.

\item[(iii)] The function $t^{2}:\hat N \to\mathbb{R}$ is globally strictly convex  on $\hat N$.
\end{enumerate}
\end{theorem}

\begin{proof}
Consider a smooth partition of unity $\phi_\sigma,\phi_h\in C^\infty((0,+\infty))$ such that
\[
0\leq\phi_\sigma(t)\leq 1,\quad\phi_\sigma|_{(0,R_1]}\equiv 1,\quad\phi_\sigma|_{[R_2,\infty)}\equiv 0,\quad\phi_\sigma' \leq 0
\]
and
\begin{equation}\label{sum}
\phi_\sigma(t)+\phi_h(t)=1,\quad\forall t\in(0,+\infty).
\end{equation}
Define 

\[
\hat{\sigma}_{il}(t,\theta^2,...,\theta^n):=\phi_\sigma(t)\sigma_{il}(t,\theta^2,...,\theta^n)+\phi_h(t)h^{(k)}_{il}(t,\theta^2,...,\theta^n),
\]
where
$$h^{(k)}_{il}(\theta^2,...,\theta^n)d\theta^{i}\otimes d\theta^{l}:= k^{-1/2} \sinh(\sqrt{k}t)(g_{\mathbb S^{n-1}})_{il}(\theta^2,...,\theta^n).$$
Note that $(\hat N^{n},\hat h)
=(\mathbb{R}^{n},dt^{2}+\hat{\sigma}_{il}(t,\theta^2,...,\theta^n)d\theta^{i}\otimes d\theta^{l})$ is a well defined $n$-dimensional Riemannian manifold and conditions $(i)$ and $(ii)$ of the statement are automatically satisfied by construction.

Consider the function $t\in C^{\infty}(\rr^n\setminus\{0\})$. By the assumption on the metric $h$, and since $(\hh^n_k,g_{\hh^n_k})=(\mathbb{R}^{n},dt^{2}+h^{(k)}_{il}(t,\theta^2,...,\theta^n)d\theta^{i}\otimes d\theta^{l})$ is a Cartan-Hadamard manifold, we have that $t$ is a convex function (strictly convex off the radial direction) on both $B_{R^1}^{\hat N^{n}}$ and $\hat N^{n}\setminus B_{R^2}^{\hat N^{n}}$. To prove the theorem we will show that $t:\hat N^{n}\to \rr$ is convex (strictly convex off the radial direction). This will imply that $t^2:\hat N^{n}\to\rr$ is strictly convex and, because of $(i)$, smooth on all of $N_\hj$. The following formula, which is a consequence of the polar structure of the metric at hand, will be applied to $(\ball,h)$, $(\hat N^{n},\hat h)$ and $(\hh^n_k,g_{\hh^n_k}).$

\begin{lemma}\label{lem_hess_rad}
Let $\ball$ be endowed with the Riemannian metric $h$ that, in polar coordinates, writes as in (\ref{metric-symm}).
Then on $(\ball\setminus\{0\},h)$ it holds
\begin{equation*}
\Hess t|_{(t,\theta^2,...,\theta^n)}(X,X)= \frac 12 \sum_{i,l=2}^nX^iX^l\frac{\partial}{\partial t} \sigma_{il}(t,\theta^2,...,\theta^n),
\end{equation*}
for all vector fields $X=X^{0}\frac{\partial}{\partial t}+\sum_{i=2}^nX^i\frac{\partial}{\partial \theta^i}$ on $\rr^n$.
\end{lemma}

\noindent According to Lemma \ref{lem_hess_rad} and recalling \eqref{sum}, we have
\begin{align}\label{hess-mix}
\Hess t|_{(t,\theta^2,...,\theta^n)}(X,X)&= \phi_\sigma(t) \left[\frac 12 X^iX^l\partial_t \sigma_{il}(t,\theta^2,...,\theta^n)\right] \\
&+ \phi_h(t) \left[\frac 12 X^iX^l\partial_t h^{(k)}_{il}(t,\theta^2,...,\theta^n)\right] \nonumber\\
&+  \frac 12\phi'_h(t) X^iX^l\left[h^{(k)}_{il}(t,\theta^2,...,\theta^n)-\sigma_{il}(t,\theta^2,...,\theta^n)\right].\nonumber
\end{align}
The first two terms on RHS of \eqref{hess-mix} are positive, since they correspond to a positive cut-off function times the hessian of $t$ on $(\ball,h)$ and $\hh^n_k$ respectively.
Since $\phi'_h\geq 0$, in order to conclude the proof it's enough to show that for $k$ large enough it holds
\[
X^iX^lh^{(k)}_{il}(t,\theta^2,...,\theta^n)\geq X^iX^l\sigma_{il}(t,\theta^2,...,\theta^n),\quad\forall (t,\theta^2,...,\theta^n)\in B_{R_2}^{\rr^n}\setminus B_{R_1}^{\rr^n}.
\]
Note that this latter makes sense since both members of the inequality do not depends on the (local) coordinate system. Since $B_{R_2}^{\rr^n}\setminus B_{R_1}^{\rr^n}$ is relatively compact, there exists a constant $c_2=c_2(R_1,R_2,h)>0$ such that for all vector fields $X$,
\[
X^iX^lh^{(1)}_{il}(t,\theta^2,...,\theta^n)\geq c_2 X^iX^l\sigma_{il}(t,\theta^2,...,\theta^n),\quad\forall (t,\theta^2,...,\theta^n)\in B_{R_2}^{\rr^n}\setminus B_{R_1}^{\rr^n}.
\]
Noticing that
\[
h^{(k)}_{il} = \frac{\sinh^2\left(\sqrt kt\right)}{k\sinh^2 t} h^{(1)}_{il},
\]
it is enough to choose $k$ in such a way that
\[
\sinh^2\left(\sqrt kR_1\right)\geq c^{-1}_2k\sinh^2 R_1.
\]
\end{proof}

Now, we will show how Theorem \ref{th_gluing_3} can be easily improved in
order to give to $N$ a Euclidean end.

\begin{theorem}
[Euclidean deformation]\label{th_gluing_4}
Let $\ball$ be endowed with the Riemannian metric $h$ that, in polar coordinates, writes as in (\ref{metric-symm}). Assume that the smooth function $t^2:\ball \to [0,R^2)$ is convex.
Fix $0<R_{1}<R$. Then, there exist
a radius $R_{4}>R_{1}$ and a manifold $(\tilde N^n,\tilde h)=(\mathbb{R}^{n},dt^{2}+\tilde{\sigma}
_{il}(t,\theta^2,...,\theta^n)d\theta^{i}\otimes d\theta^{l})$ such that:

\begin{enumerate}
\item[(i)] $\tilde{\sigma}_{il}(t,\theta^2,...,\theta^n)d\theta^{i}\otimes d\theta^{l}
={\sigma}_{il}(t,\theta^2,...,\theta^n)d\theta^{i}\otimes d\theta^{l}$ when $t<R_1$.

\item[(ii)] $\tilde{\sigma}_{il}(t,\theta^2,...,\theta^n)d\theta^{i}\otimes d\theta^{l}=t^2g_{\mathbb S^{n-1}}(\theta^2,...,\theta^n)$ when $t \geq R_4$.

\item[(iii)] The squared distance function $t^{2}:\tilde N^n\rightarrow\mathbb{R}$
is smooth and strictly convex on
$\tilde N^n$.
\end{enumerate}
\noindent In particular $\tilde N^n\setminus B_{R_4}^{\tilde N^n}$ is isometric to $\rr^n\setminus\ballr{R_4}$.
\end{theorem}
\begin{proof}
Let $\hat N^{n}$ be the manifold obtained thanks to Theorem \ref{th_gluing_3} and set $\hat{\sigma}(t):= k^{-1/2} \sinh
(\sqrt{k}t)$. Fix $R_{3}>R_{2}$ and $R_{4}>\hat\sigma(R_3)$. We construct a new smooth increasing function $\tilde{\sigma}%
:[R_{2},\infty)\rightarrow\mathbb{R}$ such that $\tilde{\sigma}(t)=\hat{\sigma}(t)$
for $R_{2}<t<R_{3}$ and $\tilde \sigma(t)=t$ for $t>R_{4}$. Finally we define
\[
\tilde{\sigma}_{il}(t,\theta^2,...,\theta^n)d\theta^{i}\otimes d\theta^{l}:=%
\begin{cases}
\hat{\sigma}_{il}(t,\theta^2,...,\theta^n)d\theta^{i}\otimes d\theta^{l}, & 0<t<R_{3},\\
\tilde{\sigma}^{2}(t)g_{\mathbb S^{n-1}}, & t>R_{2}.
\end{cases}
\]
Consider the function $t^{2}:\tilde N^{n}\rightarrow\mathbb{R}$. We already know from Theorem \ref{th_gluing_3} that $t^{2}$ is strictly convex when $t(x)<R_3$. On the other hand, since $\tilde{\sigma}$ is a
smooth increasing function on $[R_{2},\infty)$, according to Lemma \ref{lem_hess_rad} we have that $t^2$ is strictly convex also for $t(x)>R_2$, hence on all of $\tilde N^{n}$. 
\end{proof}

\section{Proof of the main theorem}
For the ease of notation, throughout all the proof we will keep
the same set of indeces each time we will extract a subsequence from a given
sequence. Moreover, given two functions $f,g:X \to \rr_{\geq 0}$ we write $f \approx g$ to mean that there exists a constant $c\geq1$ such that $c^{-1} \cdot f(x) \leq g(x) \leq c \cdot f(x)$, for all $x \in X$. Let us begin by observing the following lemma, which summarizes some useful standard results from the calculus of variations (see for instance \cite{Da}).

\begin{lemma}\label{lemma_propertyenergy}
Let (M,g) be a compact $m$-dimensional Riemannian manifold possibly with boundary $\partial M \not=\emptyset$ and let $h$ be a complete Riemannian metric on $\rr^n$ such that $c^{-1}\cdot g_{\rr^n}\leq h \leq c\cdot g_{\rr^n}$ in the sense of quadratic forms for some constant $c\geq1$. For every smooth map $u : (M,g) \to (\rr^n ,h)$ consider the energy functional $E_2 (u) = \frac{1}{2} \int_M \| du \|^2_{\mathrm{HS}} d\vol$ where, in local coordinates,
$\| du \|^2_{\mathrm{HS}}= g^{\alpha \beta}(x)\partial_\alpha u^i \partial_\beta u^j h_{ij}(u).$
Then:
	\begin{enumerate}
	\item [(a)] $E_2$ is well defined and finite on $W^{1,2}(M,\rr^n)$.
	\item [(b)] $E_2$ is continuous on $W^{1,2}(M,\rr^n)$.
	\item [(c)] $E_2$ is lower semicontinuous with respect to the weak convergence in $W^{1,2}(M,\rr^n)$.
	\end{enumerate}
\end{lemma}

Thanks to Proposition 2.3 in \cite{PV-pDirichlet-compact}, without loss of generality we can suppose that $f\in \Lip(M,N)$.  We fix $R_1<R$ such that $f(M)\subset B^N_{R_1}$ and we apply Theorem \ref{th_gluing_4} to get the complete manifold $(\tilde N^n, \tilde h)$ which is diffeomorphic to $\rr^n$ and Euclidean outside a larger ball $B_{R_4}$ (from now on it is understood that balls are considered with respect to the metric $\tilde h$).  Since the smooth metric $\tilde h$ of
$\tilde N^n$ is Euclidean outside $B_{R_{4}}$, we have $ \dist_{\rr^n} \approx \dist_{\tilde N^n}$ and,
in particular, $f \in \mathrm{Lip}(M,\tilde N^n)$. Moreover, according to Lemma \ref{lemma_propertyenergy}, for every $u \in W^{1,2}_\mathrm{in} (M,\tilde N^n)$ we have $\tilde{E}_{2}(u) < +\infty$ where
\[
\tilde{E}_{2}(u)=\frac{1}{2}\int_{M}\|du\|_{\mathrm{HS}(M,\tilde N^n)}^{2}d\vol.
\]
Therefore the set $W^{1,2}_\mathrm{in} (M,\tilde N^n,f)$ is non-empty and
\[
\mathcal{I}_{f}:=\inf_{u\in W^{1,2}_\mathrm{in} (M,\tilde N^n,f)}\tilde{E}_{2}(u)<+\infty.
\]
Let $\{v_{j}\}_{j=1}^{\infty}\subset W^{1,2}_\mathrm{in} (M,\tilde N^n,f)$ be a sequence minimizing
the energy in this Sobolev space, i.e. $\tilde{E}_{2}(v_{j})\rightarrow\mathcal{I}_{f}$
as $j\rightarrow\infty$. Choosing a subsequence if necessary, we can suppose
$\tilde{E}_{2}(v_{j})<2\mathcal{I}_{f}$ for all $j$. Since $\Vert Dv_{j}\Vert_{L^{2}(M,\mathbb{R}^{n})}<2c\mathcal{I}_{f}$
for all $j$ and, moreover, all maps in $W^{1,2}_\mathrm{in} (M,N_{\tilde\sigma},f)$ have the same
boundary datutm $f$, a uniform Poincar\'{e} inequality holds. 
Namely, there exists a constant $C=C(M,f,\mathcal{I}_{f})>0$
independent of $j$ such that
$\|v_j\|_{L^2(M,\rr^n)} \leq C$,
for all $j$. Then $\{v_{j}\}_{j=1}^{\infty}$ is bounded in $W^{1,2}%
(M,\mathbb{R}^{n})$ and, up to choosing a subsequence, $v_{j}$ converges to
some $v\in W^{1,2}(M,\mathbb{R}^{n})$ weakly in $W^{1,2}(M,\mathbb{R}^{n})$.
Since $M$ is compact, $\{v_{j}\}_{j=1}^{\infty}$ is bounded in $W^{1,p^{\prime
}}(M,\mathbb{R}^{n})$ for some $1<p^{\prime}\leq 2$. By the Kondrachov theorem, $v_{j}$ converges
strongly in $L^{s}(M,\mathbb{R}^{n})$ for any $1<s<(mp^{\prime})/(m-p^{\prime
})$ and hence pointwise almost everywhere. Since $v_{j}-f\in W_{0}%
^{1,2}(M,\mathbb{R}^{n})$ for all $j$, the weak limit $v-f\in W_{0}%
^{1,2}(M,\mathbb{R}^{n})$.
Then, by property (c) of Lemma \ref{lemma_propertyenergy},
\begin{equation}
\tilde{E}_{2}(v)\leq\liminf_{j\rightarrow\infty}\tilde{E}_{2}(v_{j}). \label{lsc}%
\end{equation}
It follows from (\ref{lsc}) and the uniform boundedness of $\tilde{E}_{2}(v_{j})$
that
\[
\mathcal{I}_{f}\leq \tilde{E}_{2}(v)\leq\liminf_{j\rightarrow\infty}\tilde{E}_{2}(v
_j)=\mathcal{I}%
_{f},
\]
so that $\tilde{E}_{2}(v)=\mathcal{I}_{f}$, i.e. $v$ minimizes $\tilde{E}_{2}$ in
$W^{1,2}_\mathrm{in} (M,\tilde N^n,f)$.\\

Now, we fix any $R_5 > R_4$ and we  claim that we can choose the minimum $v$ for $\tilde{E}_{2}$ in
$W^{1,2}_\mathrm{in} (M,\tilde N^n,f)$ such that $$v(M)\subset \bar B_{R_5}.$$  Indeed, this follows by replacing a given minimum with its radial projection onto the ball $\bar{B}_{R_5}$ as shown in the next
\begin{lemma}
Let (M,g) be a compact manifold possibly with boundary $\partial M \not=\emptyset$ and let $h$ be a complete Riemannian metric on $\rr^n$ such that $h = g_{\rr^n}$ on $\rr^n \setminus B^{\rr^n}_R$. Set $E_2 (u) = \frac{1}{2} \int_M \| du \|^2_{\mathrm{HS}} d\vol$. Fix $R' > R$ and let $\Pi : \rr^n \to \rr^n$ be the radial projection onto $B^{\rr^n}_{R'}$, namely, $\Pi(x) = R' x/|x|$ if $|x| \geq R'$ and $\Pi(x) = x$ if $|x|<R'$. Then, for any $v \in W^{1,2}(M,\rr^n)$,
\[
E_2(\Pi(v)) \leq E_2(v).
\]
\end{lemma}

\begin{remark}
\rm{
Note that, if $v \in W^{1,2}(M,\rr^n,f)$ with $f(M) \subset \bar{B}^{\rr^n}_R$, then also $\Pi(v) \in W^{1,2}(M,\rr^n,f)$.
}
\end{remark}

\begin{proof}
Let $w_{k}:M \to \rr^n$ be a sequence of smooth maps satisfying
\[
w_{k}\rightarrow v \text{ in }W^{1,2}(M,\rr^n) \text{ and pointwise a.e on } M.
\]
It follows by property (b) of Lemma \ref{lemma_propertyenergy} that
\begin{equation}\label{strongconv}
E_2(w_k) \to E_2(v).
\end{equation}
Let $\bar{w}_{k}:=\Pi\circ w_{k}$. Since $\Pi$ is $1$-Lipschitz, the sequence $\bar{w}_{k}$ is still bounded in $W^{1,2}(M,\rr^n)$ and, therefore, up to passing to a subsequence,
\[
\bar{w}_k \rightharpoonup \bar{w}_\infty \text{ in } W^{1,2}(M,\rr^n).
\]
On the other hand, as above, using Kondrakov theorem we see that (a subsequence of) $\bar{w}_k$ converges pointwise a.e. to $\bar{w}_\infty$. Therefore, by continuity of $\Pi$ and uniqueness of the pointwise limit, we have that
\[
\bar{w}_\infty = \Pi(v) \text{ a.e. on }M.
\]
In particular
\[
\Pi (v) (M) \subset \bar B^{\rr^n}_{R'}
\]
and, by property (c) of Lemma \ref{lemma_propertyenergy},
\begin{equation}\label{weakconv}
\liminf_{k \to +\infty} E_2(\bar{w}_k)  \geq E_2( \Pi(v)).
\end{equation}
To conclude, let us observe that, by a direct computation of the weak derivatives of $\bar w_k$,
\begin{equation}\label{energy-proj}
E_2(\bar w_k) \leq E_2(w_k).
\end{equation}
Indeed,
\[
\|d\bar w_k\|_{\mathrm{HS}}(x)=\|d w_k\|_{\mathrm{HS}}(x) 
\]
if $w_{k}(x)\in B_{R'}(0)$, while
\[
g(D\bar{w}_{k}^{i},D\bar{w}_{k}^{j})(x) h_{ij}\circ \bar w_k(x)
\leq g(Dw_{k}^{i},Dw_{k}^{j})(x) h_{ij}\circ w_k(x)
\]
if $w_{k}(x)\not\in
B_{R'}^{\rr^n}(0)$ because, in this set, $ h$ is the Euclidean metric. Therefore, taking the liminf in \eqref{energy-proj}, and using \eqref{strongconv} and \eqref{weakconv}, we obtain
\[
E_2(v) = \lim_{k\to + \infty} E_2(w_k)
	= \liminf_{k \to +\infty} E_2(w_k)
	\geq \liminf_{k \to +\infty} E_2(\bar{w}_k)
	\geq E_2(\Pi(v)),
\]
as required.

\end{proof}

Now, since the minimum $v$ is bounded, thanks to Corollary \ref{th_equiv-sob_3}, we have $v \in W^{1,2}_{\mathrm{ex}}(M,B_{2R_5},f)$. Actually, the above arguments show that $v$ minimizes the energy $\tilde{E}_2$ in this space. Indeed, if
$w \in W^{1,2}_{\mathrm{ex}}(M,B_{2R_5},f)$ satisfies $\tilde{E}_2(w) < \tilde{E}_2(v)$, then by Corollary \ref{th_equiv-sob_3}, $w \in W^{1,2}_{\mathrm{in}}(M,\tilde N,f)$, this would contradict the minimizing property of $v$ in the space $W^{1,2}_{\mathrm{in}}(M,\tilde N,f)$. Thus, we can apply Schoen-Uhlenbeck regularity theory with target the relatively compact, open subset $B_{2R_5}$ in the complete manifold  $\tilde N^n$ and conclude that $v \in C^{\infty}(\mathrm{int}M,B_{2R_5}) \cap C^0(M,B_{2R_5})$. As a matter of fact, since $\tilde N^n$ is Euclidean outside $B_{2R_5}$ and $v \in W^{1,2}_{\mathrm{in}}(M,\tilde N^n)$ satisfies the intrinsic harmonic mapping system, the boundary regularity can be also obtained directly from \cite{JM-MathAnn}. \\

Finally, as we just observed, $v$ is minimizing in $\tilde N^n$ therefore it is harmonic. Since the composition of the harmonic map $v$ with the squared distance function from $q_0$ in $\tilde N^n$ is subharmonic, an application of the standard maximum principle yields that $v(M) \subset B_{R_1} = B^N_{R_1}(q_0)$ and, therefore, it is a solution of the original Dirichlet problem.\\

The proof of the theorem is completed.

\begin{remark}
\rm{
As it is clear from the above proof, the solution to the Dirichlet problem obtained in Theorem \ref{main_th} is a smooth map minimizing the energy (with fixed boundary data), at least in case the squared distance function $\varrho$ from the origin of the ball is smooth and strictly convex on $B^N_{R+\varepsilon}(q_0)$ for some $\varepsilon>0$. In this direction, the convexity of the distance function seems to be natural, at least in high dimensions. In fact, consider a $n$-dimensional model manifold $N^n_{\sigma}$ with warping function $\sigma$ as in \eqref{model2}, and suppose that $a_0>0$ is the first radius at which $\sigma'(a_0)=0$, that is, the first radius at which $\varrho$ stops being strictly convex. Then, generalizing a result of Jager and Kaul into the sphere, \cite{JK,Hong}, Tachikawa, \cite{Tachi-manuscripta}, showed that the equator map $\psi: B_1^{\R^n}\to N_\sigma$ defined by $\psi (x)=(a_0,x/|x|)$ is a non regular minimizer when $0<-\sigma(a_0)\sigma''(a_0)<\frac{(n-2)^2}{4(n-1)}$ (while it is unstable otherwise). It would be interesting to understand whether this bound on the dimension has a geometric meaning and, in this case, whether Tachikawa's result can be generalized to the nonsymmetric setting. As for the first of these problems, note that at the point $a_0$, we have $-\sigma(a_0)\sigma''(a_0)=\sect_{\mathrm{rad}}/\sect_{\mathrm{tang}}$. Therefore, Tachikawa's condition looks like a pinching condition on these curvatures.

By the way, the same paper \cite{Tachi-manuscripta} contains also a proof of our main Theorem in case both the target manifold and the initial datum are rotationally symmetric.
}
\end{remark}

\appendix
\section{Analysis of known results and their proofs} \label{appendix-known}
Let $(M,g)$ be a compact, connected, $m$-dimensional Riemannian manifold with
smooth boundary $\partial M\neq\emptyset$ and let $(N,h)$ be a complete,
connected, $n$-dimensional Riemannian manifold without boundary.

In their 1977 seminal paper \cite{HKW-Acta}, S. Hildebrandt, H. Kaul and K.-O. Widman proved the existence of a solution $u\in C^{2}(\mathrm{int}M,N)\cap
C^{0}\left(  M,N\right)  $ of the Dirichlet problem for a boundary datum $f\in
C^{1}(\partial M,N)$ with sufficiently small range. More precisely, they
considered the case where $f(\partial M)$ is inside a \textit{uniformly
regular ball} $B_{R}^{N}(q_{0})$ of $N$. This means that the following
conditions are both satisfied:

\begin{enumerate}
\item[(a)] $\bar{B}_{R}^{N}(q_{0})\cap\mathrm{cut}(q)=\emptyset$ for
every $q\in\bar{B}_{R}^{N}(q_{0})$

\item[(b)] $R \sqrt{\Lambda}<\pi/2$, where $\Lambda \geq 0$ denotes an upper bound for the sectional curvatures of $N$ on $B_{R}^{N}(q_{0})$.
\end{enumerate}

\noindent In case condition (a) is required only at $q=q_{0}$  then the ball is
simply called \textit{regular}. This makes it clear what uniformity in the
above definition means. We also stress that, according to condition (a), every
pair of points $q_{1},q_{2}\in\bar{B}_{R}^{N}(q_{0})$ are connected by a
unique minimizing geodesic of $N$ \ whereas, by Hessian comparison, condition
(b) implies that the distance function $r(q)=d^{N}\left(  q,q_{1}\right)  $
from any point $q_{1}\in B_{R}^{N}(q_{0})$ is strictly convex on $B_{R}
^{N}(q_{0}) \cap B^N_{R}(q_1)$. Actually, the ball itself is strongly convex in the sense that
every pair of its points are connected by a unique minimizing geodesic of $N$
and this geodesic lies inside the ball.\medskip

The proof provided in \cite{HKW-Acta} consists of the following steps:\\

\noindent(i) Using the direct calculus of variations it is proved that the
energy functional $E_{2}$ has a minimizer among all the (intrisic)
$W_{\mathrm{in}}^{1,2}$-maps $v:M\rightarrow N$ subjected to the
boundary-trace condition $\left.  v\right\vert _{\partial M}=$ $f$, and whose
range $v(M)$ is confined in the closed ball $\bar{B}_{R}^{N}\left(
q_{0}\right) $ of $N$. Following the terminology introduced later in
\cite{Fu1-Annali}, we can speak of a constrained minimization problem. To
this purpose, all that is needed is that the compact ball $\bar{B}_{R}%
^{N}\left(  q_{0}\right) $ lies within the cut-locus of $q_{0}$, i.e.,
$R<\mathrm{inj}^{N}(q_{0})$. This allows the authors to define the Sobolev
space $W^{1,2}$ intrinsically.\\

\noindent(ii) Next, by assuming that the ball $B_{R}^{N}\left(  q_{0}\right)  $ is
regular, it is shown that the function $v(x)=d_N(x,q_0)^2$ satisfies $\Delta v\geq0$ in the sense of distributions. Hence, in case $f\left(
\partial M\right)  \subset B_{R}^{N}(q_{0})$, by the usual maximum principle,
$u(M)\subset B_{R}^{N}(q_{0})$. In particular, a-posteriori, $u$ satisfies the
harmonic mapping system, i.e., the Euler-Lagrange equations associated to the
energy functional. A slightly different argument, later extended by M. Fuchs to the $p$-harmonic setting (see below), is provided in \cite{H-LNM}.

It is worth to remark that, a-priori, the range of the minimizer $u$ produced
in step (i) could touch $\partial B_{R}^{N}(q_{0})$ and, therefore, it is not
possible to assume from the beginning that $u$ is a weak solution of the harmonic mapping system.
As a consequence, although the squared distance function $d_{N}(\cdot,q_{0})^{2}$ is
convex, one can not deduce a-priori the above subharmonicity property using
the composition law of the tension fields. Instead, the authors combine the
minimization property of $u$ with a careful comparison argument that implies
the quadratic-form estimate
\begin{equation}
Q_{y}(v,v)=h(v,v)-\sum_{k}y^{k}\Gamma_{ij}^{k}(y)h(v^{i},v^{j})\geq
0\label{intro-quadratic}%
\end{equation}
for every $y\in B_{R}^{N}\left(  q_{0}\right)  $ and for every $v^{1}%
,...,v^{m}\in T_{y}N$. We stress that the curvature assumption is vital for
the method to work and, as already alluded to above, it cannot be replaced by
the mere convexity of the distance function.\\

\noindent(iii) The final step consists in proving that $W_{\mathrm{in}}^{1,2}$-weak solution $u$ of the
(intrinsic) harmonic mapping system, has the expected regularity. As for the
interior regularity, the crucial point is to show that  $u$ is continuous and
this is achieved using somewhat Euclidean methods that require the uniformity
condition (b) in the definition of uniformly regular ball. Indeed, they need to use the smooth convexity inside the whole ball $B_R^N(q_0)$ of the squared distance from different origins. From the viewpoint of the regularity of the minimizer, and in view of the Schoen-Uhlenbeck theory, this condition sounds a little bit stronger than necessary. However, note that the authors are dealing with regularity
of weak solution of differential system rather than a minimizer and, therefore, it is reasonable that their assumptions are stronger.\\

In 1990, as a consequence of his constrained regularity theory, M. Fuchs, \cite{Fu2-Annali}, extended
Hildebrandt-Kaul-Widman result to the Dirichlet problem for $p$-harmonic maps,
namely, to critical points of the $p$-energy functional $E_{p}(v)=\frac
{1}{p}\int_{M}\| dv\| ^{p}_{\mathrm{HS}} d\vol$, $p \geq 2$. A comparison of the
proofs shows that, although Fuchs follows essentially the strategy of
\cite{HKW-Acta}, he introduced some interesting variations that we are going
to discuss.\\

\noindent(i) The minimizer $u$ of $E_{p}$ for the constrained minimization
problem is obtained in the class $W_{\mathrm{ex}}^{1,p}$ of extrinsic Sobolev maps. Next, the corresponding regularity theory developed by
Fuchs in \cite{Fu1-Annali, Fu2-Annali} is applied and gives that $u\in C^{0}(M,N)\cap
C^{1}(\mathrm{int}M,N)$. As a matter of fact:\\
	
	(i.a) Fuchs improves the full interior regularity result in \cite{HKW-Acta, H-LNM} in that he does not need the ball to be uniformly regular. Actually, Fuchs extends to the case $p>2$ the full regularity theory by Schoen-Uhlenbeck allowing also constrained minimizers. See \cite{DF-Manuscripta} for the case $p=2$.\\
		
	(i.b) In a key step of the proof of the interior regularity, the author implicitly identifies extrinsic and intrinsic Sobolev maps in order to deduce, using the curvature assumption, that a given constrained ``Minimizing Tangent Map" (MTM in the sense of Schoen-Uhlenbeck) is constant, via a certain integral inequality. As we have seen in Section \ref{Sobolev-spaces}, such an identification, although non-trivial, is absolutely correct. We also observe that, regardless of any regularity question, proving the constancy of a constrained MTM is non-trivial even if the range of the map is inside a convex supporting ball: a-priori, the minimizer does not satisfies any differential equation or inequality that can be coupled directly with a convex function to get the desired vanishing. As already remarked above, the desired vanishing is obtained using the curvature condition in a crucial way.\\

\noindent (ii) Next, using the assumption $f\left(  \partial M\right)  \subset
B_{R}^{N}(q_{0})$, it is proved that $u(M)\subset B_{R}^{N}(q_{0})$ and,
therefore, that $u$ is a weak solution of the $p$-harmonic mapping system. The
idea, following \cite{H-LNM}, is to use the quadratic-form inequality
(\ref{intro-quadratic}) joint with the minimization property of $u$ to get the
integral inequality for $v(x)=d_{N}^2(u(x),q_{0})$
\begin{equation}
\int_{\left\{  x\in M:v(x)\geq R^{2}\right\}  }h( \|du\|_{\mathrm{HS}} ^{p-2}\nabla v,\nabla v)  \leq0.\label{intro-integral}
\end{equation}
It is worth to observe that this inequality is strictly weaker than the fact
that $v$ is a distributional solution of
\[
\operatorname{div}(\|du\|_{\mathrm{HS}} ^{p-2}\nabla v)\geq0\text{, on }\mathrm{int}M.
\]
This fact would imply
immediately the desired conclusion according to the maximum principle obtained
in \cite{PV-pDirichlet-symmetry}. Moreover, inspection of the proof shows that (as in \cite{H-LNM})
the validity of (\ref{intro-integral}) would not require that $u$ is $C^{1}$
and would work perfectly even in the $W^{1,p}$-regularity assumption of $u$
provided the Sobolev space was understood in the intrinsic sense.

In any case, after
all, even in Fuchs arguments the curvature condition cannot be replaced by a
convexity assumption of the distance function.\\

To conclude this appendix, let us recall some facts about parabolic methods. As stated in the introduction, the nonpositive curvature of the target is usually required to get global existence of the harmonic maps flow. Ding and Lin weakened the curvature restriction, obtaining a global existence result for the heat flow, and its convergence to a harmonic map, in case the target $N$ is a closed manifolds with universal cover supporting a quadratic convex function and $\partial M=\emptyset$, \cite{DL-JPDE}. A generalization of their result to the flow with a prescribed boundary datum on $\partial M\neq\emptyset$ would give a parabolic proof of Theorem \ref{main_th} in the special case where the target geodesic ball comes from a closed manifold with universal cover supporting a strictly convex function. However, this seems not to be the case for a general convex ball. 

By the way, an attempt to generalize the parabolic techniques introduced in \cite{DL-JPDE} to the boundary case has been done in \cite{DL2-JPDE}. However, in their result the authors need once again N to be nonpositively curved.

\section{The intrinsic metric-space approach}\label{appendix-metricspaces}
In recent years, starting from the seminal papers by M. Gromov and R. Schoen, \cite{GS-IHES}, N. Korevaar and R. Schoen, \cite{KS-CAG}, and J. Jost, \cite{Jo, Jo1}, a completely intrinsic approach to Sobolev spaces and corresponding trace theory for maps $u:M \to X$ from an $m$-dimensional Riemannian manifold $(M,g)$ into  a complete, separable metric space $(X,d)$ has been developed. The definition of the Sobolev class is quite technical and involves a suitable energy functional which is defined  in terms of the distance function in the target via a limiting procedure. More explicitly, for $p=2$, letting
\[
e_{\varepsilon}(u)(x) = \frac{m}{ \area(\partial B_{\varepsilon}(x))} \int_{\partial B_{\varepsilon}(x)}\frac{d^2(u(x),u(y))}{\varepsilon^2} d\area(y)
\]
with $x \in \interior M$ and $0 < \varepsilon \ll 1$, it is said that $u$ has finite $2$-energy if
\[
\sup_{{f \in C_c(\interior M), 0 \leq f \leq 1}} \left( \limsup_{\varepsilon \to 0} \int_{\interior M} f(x)e_{\varepsilon}(u)(x) d \vol(x) \right) < +\infty.
\]
When this happens,  one has the weak convergence of measures
$e_{\varepsilon}(u) d \vol \rightharpoonup e(u) d \vol$
and, by definition, the Radon-Nikodyn derivative $e(u)$ of the limit is the energy density $ |du|^2$ of $u$.
In the same foundational papers, the authors develop the theory of minimizing and harmonic maps in this general setting. In particular, \cite{KS-CAG} and its companion \cite{Se-CAG} contain a full treatment of the $\Lip$-interior regularity theory and the $C^{\alpha}$-boundary regularity theory for energy-minimizing maps into a simply connected target of non-positive curvature (in the sense of Alexandrov). Whence, in \cite{KS-CAG}, the authors obtain a complete solution of the homotopy Dirichlet problem for any datum $f \in C^{\alpha}(M,X) \cap W^{1,2}(M,X)$ and a compact target $(X,d)$ whose universal covering has non-positive curvature. The map is a global minimizer and it is unique by the convexity of the energy  in non-positive curvature contexts. Clearly, the case where $(X,d)$ itself is simply connected of non-positive curvature is included in the picture.\\

Meanwhile, a parallel development of Sobolev spaces of maps and corresponding energy functionals was proposed by Yu. G. Reshetnyak, \cite{Re-Siberian, Re-Siberian-1, Re-Siberian-2}. His viewpoint is completely different from Korevaar-Schoen's and puts the emphasis on the fact that a map $u \in W^{1,p}(M,X)$ well behaves with respect to the composition with any $\Lip$-function $\varphi:X \to \rr$. Roughly speaking, the norm of the gradient of $\varphi \circ u$ should be dominated a.e. by $\Lip(\varphi)$ times the norm of the differential of $u$, where $\Lip(\varphi)$ denotes the Lipschitz constant of the function. When $\varphi(z) = d(z,z_0)$ is the $1-\Lip$ distance function from a point $z_0$ one therefore requires that $| \nabla (\varphi \circ u(x)) | \leq  \omega(x)$, for some $\omega \in L^p(M)$ independent of the reference point $z_0$ and defines the energy density $|d u(x)| \in L^p(M)$ as the smallest of such functions $\omega$. The corresponding energy functional turns out to be weakly lower-semicontinuous on $W^{1,p}(M,X)$ and the theory of energy-minimizing maps can be carry over. In fact, it was observed in \cite{Re-Siberian-1} that this definition of Sobolev maps is just an equivalent manifestation of Korevaar-Schoen's, that is, each of the defined energies is controlled by a multiple of the other. Therefore one can take advantage of both these viewpoints depending on the problem at hand.\\

It is worth to mention a third approach proposed by P. Hajlasz and J. Tyson, \cite{HT-Michigan}, that, in some sense, represents a metric version of the manifold-valued extrinsic definition of Sobolev space of maps. This time, one has the dual $V^{\ast}$ of a separable Banach space $V$ and assumes that the target metric space $(X,d)$ is isometrically embedded into $V^{\ast}$. When $X$ is separable, we can always take $V = \ell^1(X)$. A map $u:M \to X$ is said to be in the $W^{1,p}$-class if its realization $\check u : M \to V^{\ast}$ is $W^{1,p}$ in the sense that  $\check u$ is $L^2$ with first weak derivatives in $L^2$ (or, equivalently, if $\check u$ is $W^{1,p}(M,V^{\ast})$ in Reshetnyak sense). It is shown that, unlike the case of non-compact manifolds, this definition does not depend on the choice of the isometric embedding. Note that, unlike the manifold-valued case, the embedding is isometric in the distance-sense.
\smallskip
 
It is natural to ask whether these purely intrinsic approaches to Sobolev maps and the consequent energy-minimizing theory can be extended to convex-supporting targets thus leading to an intrinsic proof of our main Theorem. It is likely that the energy-nonincreasing Lip-projection procedure can be made formal in Reshetnyak setting and, since Reshetnyak Sobolev maps are indeed Korevaar-Schoen Sobolev maps, one can try to use the trace and regularity theory of \cite{KS-CAG, Se-CAG}. However, switching from one of these energies  to the other one is not without pain because these energies are just equivalent, not equal. Moreover, concerning both the existence and the regularity issues, inspection of the arguments supplied by Korevaar-Schoen shows that the authors use in a crucial way both the curvature condition and the convexity of the target. This is visible, for instance, in their use of the convexity of the energy functional to get the existence of a minimizer $u$, \cite[Theorem 2.2]{KS-CAG}, and in the derivation of the (weak) subharmonicity property the function $d^2(u(x),x_0)$ in $\interior M$, a fact that is crucial to prove regularity, \cite[Section 2.4]{KS-CAG}, \cite[Section 2]{Se-CAG}. Actually, in light of the computations in the smooth case, the metric viewpoint could lead to a purely  intrinsic proof of an existence and regularity result when the range of the boundary datum is inside a convex, regular ball, where the curvature is still understood in the Alexandrov sense. Regularity results in this direction are obtained in \cite{Fg-CalcVar}. 
\smallskip

In conclusion it seems that the metric approach, in its present form, cannot be applied directly to obtain an intrinsic proof of our main result. We guess that, with some further effort, it could lead to a  proof (with really intrinsic Sobolev space) of Hildebrandt-Kaul-Widman result, i.e., when the target is a convex, regular ball. However, even this conclusion would be less general than our result when applied to smooth target, as it is already observed several times in the previous sections.\\\bigskip

\textbf{Acknowledgement.}
The authors were partially supported by the \textit{Gruppo Nazionale per l'Analisi Matematica, la Probabilit\`a e le loro Applicazioni (GNAMPA)}. The first author was also partially supported by the Italian project PRIN 2010-2011 ``Variet\'a reali e complesse: geometria, topologia e analisi armonica". Most of this research was done during the visit of the first author to the Laboratoire Analyse G\'eom\'etrie et Applications in the Department of Mathematics of the University Paris 13. He is deeply grateful to this institution for the financial support and for the warm hospitality. Finally, the authors would like to thank Batu G\"uneysu for a careful reading of the first draft of the manuscript.

\end{document}